    \def\year{2020}\relax
    \documentclass[letterpaper]{article}
    \usepackage{aaai20}
    \usepackage{dsfont}
    \usepackage{times}
    \usepackage{helvet}
    \usepackage{courier}
    \usepackage{url}
    \usepackage{graphicx}
\newcommand{\citet}[1]{\citeauthor{#1} \shortcite{#1}}
\newcommand{\citep}{\cite}
\usepackage{amsfonts}       % blackboard math symbols
\usepackage{amsmath}
 \usepackage{amsthm}
\usepackage[linesnumbered,ruled,vlined]{algorithm2e}
\newtheorem{theorem}{Theorem}
\newtheorem{lemma}[theorem]{Lemma}

\newtheorem{definition}[theorem]{Definition}

\newtheorem{assumption}{Assumption}
\newtheorem{corollary}[theorem]{Corollary}

\def \E{\mathbb{E}}
\def \J{\mathbb{J}}
\def \PP{\mathbb{P}}
\def \R{\mathbb{R}}

\def \cA{\mathcal{A}}

\def \cP{\mathcal{P}}

\def \cX{\mathcal{X}}

\def\reff#1{{\rm(\ref{#1})}}
\def\be{\begin{eqnarray}}
\def\ee{\end{eqnarray}}
\def\be*{\begin{eqnarray*}}
\def\ee*{\end{eqnarray*}}
\def\beq{\begin{equation}}
\def\eeq{\end{equation}}

\def \Tb{\mathbb{T}}

\def \eps{\mathbb{\epsilon}}

\title{
On the Convergence of Model Free Learning in Mean Field Games
}

\author{
Romuald Elie,\textsuperscript{\rm 1}
Julien P{\'e}rolat,\textsuperscript{\rm 2}
Mathieu Lauri{\`e}re,\textsuperscript{\rm 3}
Matthieu Geist,\textsuperscript{\rm 4}
Olivier Pietquin\textsuperscript{\rm 4}
\\
\textsuperscript{\rm 1}Universit{\'e} Paris-Est,
\textsuperscript{\rm 2}Deepmind\\
\textsuperscript{\rm 3}ORFE, Princeton University,
\textsuperscript{\rm 4}Google Research, Brain Team\\
romuald.elie@univ-mlv.fr, perolat@google.com, lauriere@princeton.edu, 
mfgeist@google.com, pietquin@google.com
}

\begin{document}

\maketitle

\begin{abstract}
Learning by experience in Multi-Agent Systems (MAS) is a  difficult and exciting task, due to the lack of stationarity of the environment, whose dynamics evolves as the population learns. In order to design scalable algorithms for systems with a large population of interacting agents (\textit{e.g.}, swarms), this paper focuses on Mean Field MAS, where the number of agents is asymptotically infinite.
Recently, a very active burgeoning field studies the effects of diverse reinforcement learning algorithms for agents with no prior information on a stationary Mean Field Game (MFG) and learn their policy through repeated experience. We adopt a high perspective on this problem and analyze in full generality  the  convergence of a fictitious iterative scheme using any single agent learning algorithm at each step. We quantify the quality of the computed approximate Nash equilibrium, in terms of the accumulated errors arising at each learning iteration step.
Notably, we show for the first time convergence of model free learning algorithms towards non-stationary MFG equilibria, relying only on classical assumptions on the MFG dynamics. We illustrate our theoretical results with a numerical experiment in a continuous action-space environment, where the approximate best response of the iterative fictitious play scheme is computed with a deep RL algorithm.
\end{abstract}

\section{Introduction}

In Multi-agent systems (MAS), several autonomous robots or agents interact and cooperate, compete or coordinate in order to complete their task. The difficult nature of the task at hand combined with the large number of possible situations imply that the agents have to learn by experience. In comparison to the single-agent case, the derivation of efficient learning algorithms in this context is difficult due to the lack of stationarity of the environment, whose dynamics evolves as the population learns  \citep{bu2008comprehensive}. This gives rise to research topics lying at the intersection of game theory and reinforcement learning. Nevertheless, in typical examples, the number of interacting agents can be very large (\textit{e.g.}, swarm systems) and defies the scalability properties of most learning algorithms. For anonymous identical  agents, a key simplification in game theory is the introduction of the asymptotic limit where the number of agents is infinite, leading to the modeling intuition behind the theory of Mean Field Games (MFG). This calls for an analysis of model free learning scheme for MAS in terms of MFG.  

MFG were introduced by~\citeauthor{MR2269875} \shortcite{MR2269875,MR2271747} and~\citet{MR2346927-HuangCainesMalhame-2006-closedLoop} in order to model the dynamic equilibrium between a large number of anonymous identical agents in interactions. Such systems encompass the modeling of numerous applications such as traffic jam dynamics, swarm systems, financial market equilibrium, crowd evacuation, smart grid control, web advertising auction, vaccination dynamics,  rumor spreading on social media, among others. In a sequential game theory setting, each player needs to take into account his impact on the strategies of the other players. Studying games with an infinite number of players is easier from this point of view, as the impact of one single player on the others can be neglected. Hereby, the asymptotic limit with infinite population size considered in MFG becomes highly relevant. A solution to a dynamic MFG is determined via the optimal policy of a representative agent in response to the flow of the entire population. A mean field (MF) Nash equilibrium arises when the distribution of the best response policies over the population generates the exact same population flow. In most cases, a MF Nash equilibrium provides an approximate Nash equilibrium for an analogous game with a finite number of players~\citep{Cardaliaguet-2013-notes,MR3134900,carmona2018probabilisticI-II}. 

 In the abundant literature on MFG, most papers consider planning problems with fully informed agents about the game operation scheme, the reward function and the MF population dynamics. Only a few contributions focus on learning problems in MFG, see e.g.~\citep{yin2010learning,cardaliaguet2018mean,Hadikhanloo-phdthesis,hadikhanloo2019finite} for model based approaches. Very recently, a rapidly growing literature intends to approximate the solution of stationary MFG in the realistic setting where agents with no prior information on the game learn their best response policy through repeated experience. These contributions restrict to a stationary setting and focus on specific Reinforcement Learning (RL) algorithms:   Q-learning \citep{guo2019learning,yang2018mean}, fictitious play \citep{mguni2018decentralisedli} or policy gradient methods \citep{subramanianpolicy}, 
 and sometimes rely  on hardly verifiable assumptions. 
 
 In this paper, we take a step back and adopt a general high perspective on the convergence of model free learning algorithms in possibly non-stationary MFG and emphasize their potential for MAS with a large number of agents. Our approach investigates how any single-agent learning algorithm can perform in an MFG setting, in order to learn a (possibly approximate) Nash equilibrium, via repeated experiences and without any prior knowledge. Namely, we quantify precisely how the convergence of model free iterative learning algorithms reduces to the error analysis of each learning iteration step, in analogy with how the convergence of RL algorithms reduces to the aggregation of repeated supervised learning approximation errors \citep{farahmand2010error,scherrer2015approximate}.
For this purpose, our approach relies on a model free Fictitious Play (FP) iterative learning scheme for repeated games, where each agent calibrates its belief to the empirical frequency of the previously observed population flows. 
The FP approach is very natural when agents are trying to learn how to play a game by experience, while interacting with others. Before a new round of experience, they need to anticipate the behavior of the other players, and FP ergodic averaging is nicely designed  for this purpose. This algorithm is typically useful for building from experience collaboration a cooperation patterns in a MAS using a decentralized learning scheme. In our framework of interest, all agents are identical (as usual in MFG), and we consistently suppose that they use the same learning scheme.  
Whenever the agents can compute their exact best response to any population flow, FP is proved to reach asymptotically a Nash equilibrium in some (but not all~\citep{shapley1964some}) classes of games, such as  first order monotone MFG~\citep{Hadikhanloo-phdthesis}. However, in a realistic setting, the agents are not able to compute their exact best response and can only attain an approximate version of it. This induces at each iteration a learning approximation error, which propagates through the FP learning scheme. 
    
The main contribution of this paper is theoretical, as we provide a  rigorous study of the error propagation in Approximate FP algorithms for MFGs, using an innovative line of proof in comparison to the standard two time scale approximation convergence results \citep{leslie2006generalised,borkar1997stochastic}. Our convergence results are derived under easily verifiable assumptions on possibly non-stationary MFG dynamics and cost, which are highly classical in the MFG literature  (namely $1^{st}$ order monotone MFGs). This allows discussing the convergence to a (possibly approximate) MF Nash equilibrium, when using any standard single-agent learning algorithm as an inner step embedded in a FP iterative scheme. Especially, our theoretical framework encompasses the convergence of RL algorithms to MFG equilibria in non stationary settings, which, as far as we know, is new in the literature. 
We illustrate our theoretical results on an authorative MFG numerical experiment on crowd congestion, where the approximate best response of the iterative FP scheme is computed with a deep RL algorithm. This provides for the first time a model free learning example on MFG in a continuous state-action environment.

\section{Background}

\paragraph{Mean Field Games.}
MFGs were introduced by ~\citeauthor{MR2269875} \shortcite{MR2269875,MR2271747} and by~\citet{MR2346927-HuangCainesMalhame-2006-closedLoop} and correspond to the asymptotic limit of a differential game, where the number of agents is infinite. Since all agents are assumed to be identical and indistinguishable, individual interactions are irrelevant in the limit and only the distribution of states matters (see ~\citep{carmona2018probabilisticI-II} for a complete overview). Most of the MFG literature is displayed in continuous time, but we choose to present our analysis in a discrete time setting in order to alleviate the presentation and emphasize the fruitful connections with the learning literature.

Finding a Mean Field Nash equilibrium boils down to identifying the equilibrium distribution dynamics of the population as well as the best response (or optimal policy) of a representative agent to this population mean field flow. Since the number of players is infinite, each agent has an infinitesimal influence on the population distribution. Yet, since all agents are rational, at equilibrium the state distribution generated by the optimal policy must coincide with the population distribution.

\noindent\textbf{Notations.} 
Let $\cX$ and $\cA$ be compact convex subsets of $\R^\ell$ and $\R^d$ respectively, which represent the state and action spaces common to every agent. Let $T>0$ be a time horizon and let $\Tb$ denote the time sequence $\{0,1,\dots,T\}$. We denote by $\cP(\cX)$ the set of probability measures on $\cX$ and by $M_T = \cP(\cX)^{T}$ the set of all possible flows of population state distributions $\mu=(\mu_0,\mu_1,...,\mu_T)$. The initial distribution of the population is an atomless measure on $\cX$ denoted by $\mu_0$. For $\mu \in M_T$, $\mu_t$ represents the distribution at time $t$ of the state occupation of the entire population. 

\noindent\textbf{State dynamics $\&$ Mean field population flow.} 
At any time $t\in\Tb$, each agent belongs to a state $x_t\in\cX$ and picks an action $a_t\in\cA$. For a sequence of actions $a:=(a_t)_{t\in\Tb}$, the dynamics of $x$ is governed by a Markov Decision Process (MDP) with (possibly non-stationary) transition density $p_t(.|x_t, a_t, \mu_t)$ parameterized by the mean field flow $\mu\in M_T$ of the population. This indexation transcribes the interactions with the other agents, through their state distribution $\mu_t$. Typically, the dynamics of $x$ is described by an equation of the form
\begin{equation}
\label{eq:dyn-general}
    x_{t+1} = x_t + b(x_t, a_t, \mu_t) + \epsilon_{t+1},
\end{equation}
where $b: \cX \times \cA \times \cP(\cX) \to \cX$ is a drift function and $(\epsilon_t)_{t \ge 1}$ is a dynamic source of noise. We stress that the mean field term $\mu_t$ represents the whole population distribution and not just the average state,  as e.g. in~\cite{yang2018mean}.

We denote by $\Pi_T$ the set of policies (or controls) $\pi: \Tb \times \cX \to \cA$ which are feedback in the state: at time $t$, an agent using policy $\pi$ while in state $x_t$ plays the action $a_t=\pi(t,x_t)$. 
The process $x$ controlled by $\pi\in\Pi_T$ is denoted $x^\pi$.

\noindent\textbf{Agent's reward scheme.} 
An infinitesimal agent starting at time $0$ in state $x_0\in\cX$  chooses a policy $\pi\in\Pi_T$ in order to  maximize the following discounted expected sum of rewards: 
\begin{equation}
\label{eq:def-general-J}
    J(x_0,\pi,\mu) := 
    \E\left[ \sum_{t = 0}^{T-1} \gamma^t r(x^{\pi}_t,\mu_t,a_t)  \right]\;, 
\end{equation}
while interacting with the population MF flow $\mu\in M_T$. At time $t$, the agent's rewards are impacted by $\mu_t$, which represents the aggregate state distribution of all the other agents (\textit{i.e.} of the whole population). Since the agents are anonymous, only the MF distribution flow $\mu$ of the states matters. 

As $\mu_0$ denotes the state distribution at time $0$, the average reward for a representative agent is given by
\begin{equation}
\label{eq:def-JJ}
\begin{split}
    \J(\pi, \mu) 
    &:= \E_{x_0 \sim \mu_0}[J(x_0, \pi, \mu)] \,,
\end{split}
\end{equation}
when this agent uses policy $\pi\in\Pi_T$, while the mean field population flow is $\mu$. 

\begin{definition}[Best response]
\label{def:BR}
A policy $\pi^*$ maximizing $\J(.,\mu)$ is called a best response of the representative agent to the MF population dynamic flow $\mu$. 
\end{definition}

\noindent\textbf{MF Nash Equilibrium.} 
While interacting, the agents may or may not reach a Nash equilibrium, whose definition, based on the previous best response policy characterization, reads as follows:

\begin{definition}[Mean Field Nash equilibrium]
\label{eq:def-NE-MFG}
A pair $(\pi^*_t, \mu^*_t)_{t\in \Tb}$ consisting of a policy and a MF population distribution flow is called MF Nash equilibrium if it satisfies
\begin{itemize}
    \item  \textbf{Agent rationality:}  $\pi^*$ is a best response to $\mu^*$;
    \item \textbf{Population consistency:} for all $t \in \Tb$, $\mu^*_t$ is the distribution of $x^*_t$, starting with distribution $\mu_0$ and controlled by policy $\pi^*$.
\end{itemize}
\end{definition}

Namely, if the mean field population flow is $\mu^*$, the policy $\pi^*$ is optimal, and if all the agents play according to $\pi^*$, the induced mean field population flow coincides with $\mu^*$. Hereby, $(\pi^*,\mu^*)$ identifies to an MF Nash equilibrium. 

Observe that reaching an MF Nash equilibrium requires the  computation of the exact best response policy, which can be difficult in practice. We are concerned with the design of an iterative learning scheme, where the available best response is partially accurate and typically approximated by RL through repeated experiences. For example, this realistic situation arises when agents are repeatedly optimizing their daily driving trajectories, without any prior information on the traffic jam dynamics. 

\section{Fictitious Play Algorithms for MFG}

 Fictitious play~\citep{robinson1951iterative} is an iterative learning scheme for repeated games, where each agent calibrates its belief to the empirical frequency of previously observed strategies of other agents, and plays optimally according to its beliefs. This constitutes its best response. Even in simple two-player games, the convergence of FP  to a Nash equilibrium is not guaranteed~\citep{shapley1964some}. However, the convergence of FP has recently been proved for some classes of MFG \citep{Hadikhanloo-phdthesis,cardaliaguet2017learning}.

 Yet, in most cases, agents do not have access to the exact best response policy $\pi^*$ but use  an  approximate version of    it instead, in the spirit of \citep{leslie2006generalised,perolat2018actor}. 
 At iteration $n$, the agent has only access to an approximate version  $\hat\pi^{n+1}$ of the best response  $\pi^{*,n+1}$ to the anticipated mean field flow $\bar\mu^{(n)}$, defined precisely in Algorithm~\ref{algo:AFP-MFG-2}. 
 \\
  At iteration step $n$, the learning scheme induces an average additional error $\ell_n$  defined as
  \begin{equation}\label{Error_Learning}
 \ell_n \;:=\;  \J(\pi^{*,(n+1)}, \bar\mu^{(n)}) - \J(\hat\pi^{(n+1)}, \bar\mu^{(n)})
 \ge 0\;,
 \end{equation}
 for $n\in\mathbb{N}$. 
 Observe that $\ell_n$ identifies to the expected loss over the entire population at step $n$, when replacing the the exact best response $\pi^{*,(n+1)}$ by the approximate policy $\hat\pi^{(n+1)}$.\\ In Section~\ref{Section_First_Order} below, we quantify the propagation of approximating errors $\ell_n$ and clarify the convergence properties of Algorithm~\ref{algo:AFP-MFG-2} for any type of learning procedure at each intermediate step. The specific setting where the approximate optimal policy is computed using single-agent RL algorithms is discussed in Sec~\ref{Section_Discussion_RL}.

\begin{algorithm}
\DontPrintSemicolon
\KwData{An initial distribution $\mu_0$; an initial flow of distributions $\bar{\mu}^{(0)}$; a number of steps $N$.}
\KwResult{A flow of distribution $\mu$ and a  policy $\pi$.}
\Begin{
  \For{$n = 0, 1, \dots N-1$}{
    Compute $\hat\pi^{(n+1)}$, the approximate response policy against $\bar \mu^{(n)}$\; 
    Compute $\hat \mu^{(n+1)}$, the mean field flow associated to $\hat\pi^{(n+1)}$ starting with $\hat \mu^{(n+1)}_0 =\mu_0$\;
    Update $\bar \pi^{(n+1)}$, the uniformly randomized policy over $(\hat\pi^{(k)})_{k=1,\ldots,n+1}$   \;    %
     Update  $\bar \mu^{(n+1)} =\frac{n}{n+1}\bar \mu^{(n)} + \frac{1}{n+1}\hat \mu^{(n+1)}$  } \;
    
  \KwRet{$(\bar\mu^{(N)},\bar\pi^{(N)})$}
  }
\caption{Approximate Fictitious Play for MFG}
\label{algo:AFP-MFG-2}
\end{algorithm}

\paragraph{Approximate Nash equilibrium} 
 At each step $n$, we denote by $\bar\pi^{(n)}$ the representative agent belief on the aggregate population policy, defined as an equally randomized version of all previous approximate best responses $(\hat\pi^{(k)})_{k=1,\ldots,n}$: for each $t \in \Tb$ and $x \in\cX$,
    $\bar\pi^{(n)}(t,x) \in \cP(\cA)$ is the probability distribution on the set of actions $\cA$ according to which the player picks uniformly at random an element of $\{\hat\pi^{(1)}(t,x), \dots, \hat\pi^{(n)}(t,x)\}$.\\ With a slight abuse of notation, we write
    $$
        J(x_0, \bar\pi^{(n)}, \bar\mu^{(n)}) := \frac{1}{n} \sum_{k=1}^{n} J(x_0, \hat\pi^{(k)}, \bar\mu^{(n)})\;, \quad n\in\mathbb{N}\;.
    $$
and modify the definition of $\J$ in \reff{eq:def-JJ} accordingly.
 Observe for later use that, by construction,  $\bar\mu^{(n)}$ defined in Algorithm~\ref{algo:AFP-MFG-2} coincides with the population MF flow induced by the policy $\bar\pi^{(n)}$. 
In order to assess the quality of  $(\bar\mu^{(n)},\bar\pi^{(n)})$ as an (approximate) MF Nash equilibrium, we introduce, for $n\in\mathbb{N}$,
 \be*
    e_n 
    &:=& \J(\pi^{*,(n+1)}, \bar\mu^{(n)})-\J(\bar\pi^{(n)}, \bar\mu^{(n)}) \,\ge\, 0\;.
 \ee*
 The exploitability $e_n$ quantifies at iteration $n$ the expected gain for a typical agent, when shifting its belief $\bar\pi^{(n)}$ to the exact best response  $\pi^{*,(n+1)}$, while interacting with the MF population flow $\bar\mu^{(n)}$. After $n$ iterations in Algorithm \ref{algo:AFP-MFG-2}, $e_n$ is a quantitative measure of the quality of $(\bar\mu^{(n)},\bar\pi^{(n)})$ as an MF Nash equilibrium. For the sake of clarification, let us introduce a more precise weaker notion of MF Nash equilibrium,  inspired by~\citep{carmona2004nash}. 

\begin{definition}[Approximate MF Nash equilibrium]
\label{def:eps-delta-NE}
    For $\epsilon >0$ and $\delta \in (0,1)$, a pair $(\pi^*_{\epsilon,\delta}, \mu^*_{\epsilon,\delta})$ consisting of a policy and a population distribution flow is called an $(\epsilon,\delta)-$MF Nash equilibrium if
    $$
        \mu_0\left( \left\{ x_0  \,;\, J(x_0, \pi^*_{\epsilon,\delta}, \mu^*_{\epsilon,\delta}) \geq J(x_0, \pi', \mu^*_{\epsilon,\delta}) - \epsilon, \forall \pi'\right\}\right) %\geq 1- \delta,
    $$
    is at least $1- \delta$, 
    and $\mu^*_{\epsilon,\delta}$ coincides with the MF distribution flow  starting from $\mu_0$, when every agent uses policy $\pi^*_{\epsilon,\delta}$.
\end{definition}

An $(\epsilon,\delta)-$MF Nash equilibrium identifies to a weak equilibrium which reveals $\epsilon$-optimal for at least a fraction $(1-\delta)$ of the population. 
We are now in position to clarify how the exploitability $e_n$  quantifies the quality of $(\bar\mu^{(n)},\bar\pi^{(n)})$  as an MF Nash equilibrium.
 
\begin{theorem}\label{thm_approximate_Nash}
If $0\le e_n\le \epsilon^2$ for some $n\in\mathbb{N}$, 
then 
$(\bar\mu^{(n)},\bar\pi^{(n)})$ is an $(\epsilon, \epsilon)$-MF Nash equilibrium in the sense of Definition~\ref{def:eps-delta-NE}. 
If $e_n$ goes to 0 as $n \to +\infty$, any accumulation point of $\bar\mu^{(n)}$ is a MF Nash equilibrium
\end{theorem}
\begin{proof}
Fix $n\in\mathbb{N}$ and assume
    $0 \le
       e_n 
    \le \epsilon^2.$ 
    Let us introduce $$\varphi(x_0) := J(x_0, \pi^{*,(n+1)}, \bar\mu^{(n)})-J(x_0, \bar\pi^{(n)}, \bar\mu^{(n)})\,\ge 0,$$  %Recalling that 
as $\pi^{*,(n+1)}$ is the best response to the MF flow $\bar\mu^{(n)}$.\\
Using Markov's inequality and the bound on $e_n$ we obtain
\begin{align*}
    \mu_0\left(\left\{ x_0 \in\cX \,:\, \varphi(x_0) \geq \epsilon \right\}\right)
    &=
    \PP_{x_0 \sim \mu_0}\left[ \varphi(x_0) \geq \epsilon\right] 
    \\
    &\leq 
    \frac{\E_{x_0 \sim \mu_0}[\varphi(x_0)]}{\epsilon} 
    \;
    = \; \frac{e_n}{\epsilon}\,, 
\end{align*}
which is smaller than $\epsilon$.
Collecting the terms and using the definition of $\varphi$, we deduce that
$$
     \mu_0\left(\left\{ x_0;  J(x_0, \bar\pi^{(n)}, \bar\mu^{(n)}) \geq J(x_0, \pi^{*,(n+1)}, \bar\mu^{(n)}) - \epsilon \right\}\right)
$$
is at least $1-\epsilon$, 
so that $(\bar\pi^{(n)}, \bar\mu^{(n)})$ is an $(\epsilon,\epsilon)$-MF Nash equilibrium.\\ The second part of the theorem follows directly. 
\end{proof}

\section{Error propagation \& Nash equilibrium approximation for first order MFG}\label{Section_First_Order}

 Since the exploitability $e_n$ identifies to a relevant quality measure of Algorithm~\ref{algo:AFP-MFG-2} after $n$ iterations, we now evaluate how the individual learning errors $(\ell_k)_{0\le k\le n}$ aggregate over $e_n$. For the sake of simplicity, we focus our discussion on $1^\text{st}$  order MFG, \textit{i.e.} without source of noise in the dynamics. This allows us to build our reasoning on the 
analysis of \cite[Chapter 3]{Hadikhanloo-phdthesis} and to avoid a limitative restriction to second order games with a potential structure, for which similar results should hold in that setting as well, see~\citet{cardaliaguet2017learning}.
 
\subsection{First order mean field game} \label{Section_FirstOrderMFG}

The state $(x_t)_t$ evolves in $\R^d$ with  dynamics~\eqref{eq:dyn-general}, where we take $b(x,a,\mu) = a$, and $(\epsilon_t)_t = 0$. In other words, each agent controls exactly its state variation between two time steps and does not endure any noise. 
While interacting with a MF flow $\mu \in M_T$, each agent intends to maximize the classical reward scheme given by~\eqref{eq:def-general-J}
with a running reward at time $t$ of the form: 
\begin{equation}
\label{eq:def-model-r}
    r(x^\pi_t, \mu_t, a_t) \mapsto \tilde r(x^\pi_t, a_t) + \bar r(x^\pi_t, \mu_t),
\end{equation}
where the extra $\bar r$ captures the impact of the other agents' positions. In Sec.~\ref{Section_Expe}, we provide in particular a congestion example where $\bar r$  models an appeal for non-crowded regions. This type of conditions translates into the so-called Lasry-Lions monotonicity condition~\citeauthor{MR2269875} \shortcite{MR2269875,MR2271747} which ensures uniqueness of MF Nash equilibrium. More precisely, existence and uniqueness of solution to the $1^\text{st}$ order MFG of interest hold under the following classical set of assumptions.

\begin{assumption}\label{Assumption_reward}
    For some constant $C$, the reward functions $\tilde r$ and $\bar r$ satisfy:
    \begin{itemize}
        \item For any $x\in\cX$, the map $\tilde r(x,.)$  is twice differentiable and 
$$
\frac{1}{C} I_d \le D_{aa} \tilde r(x,.) \le C I_d\;,
$$
 \item The function $\bar r$ is continuous on $: \cX\times\cP(\cX)$ and $\bar r(.,m)$ is $\mathcal{C}^1$ on $\cX$\;,
\item We have 
$$
 \|\tilde r(.,.)\|_\infty + \|\bar r(.,.)\|_\infty \le C\,,
$$
\item The Lasry-Lions monotonicity condition holds: for all $m_1,m_2 \in \mathcal{P}(\cX)$, 
 \beq\label{monotone} 
 \int_{\cX} \left[\bar r(.,m_1)-\bar r(.,m_2)\right] d[m_1 - m_2]< 0 \;.
\eeq
    \end{itemize}
\end{assumption}
\subsection{Error propagation in the Fictitious Play algorithm} \label{Sec_Error_Propagate}
We now investigate how the learning error $(\ell_n)_n$ propagates through FP for any learning algorithm, while Sec~\ref{Section_Discussion_RL} focuses on the specific case where the best response is approximated via RL.\\
The key ingredient of FP iterative learning schemes is the quick stabilization of the sequence of beliefs  $(\bar\mu^{(n)})_n$. 

\begin{lemma}\label{Lemma_close_FP}
Under Assumption~\ref{Assumption_reward}, the FP MF flow $\bar\mu^{(n)}$ satisfies: 
$$ 
    d_1(\bar\mu^{(n)}, \bar\mu^{(n+1)})\le \frac{C}{n}\;, \;\; n\in\mathbb{N}, \quad \mbox{for some  } C>0\,,
$$
where $d_1$ is the Wasserstein distance. 
\end{lemma}
The proof follows from a straightforward adaptation  of ~\cite[Lemma 3.3.2]{Hadikhanloo-phdthesis} to our setting. 

As the sequence of beliefs $\bar\mu_n$ stabilizes, the impact of recent learning errors $(\ell_n)$ reduces and we are in position to quantify the global error $e_n$ of the algorithm after $n$ iteration steps. This is the main result of the paper, whose proof interestingly differs from the more classical two-time scale approximation argument~\citep{borkar1997stochastic}. 
\begin{theorem}\label{Thm_error_1} Under  Assumption~\ref{Assumption_reward}, the Nash equilibrium quality $(e_n)_n$ 
satisfies both estimates: 
 for all $n \in \mathbb{N}$
\begin{align}
    e_n 
    & \le  
     \frac{C_1}{n}+\frac{C_1}{n}\sum_{i=1}^n d_1(\mu^{*,(i+1)},\hat\mu^{(i+1)}) + \frac{1}{n}\sum_{i=1}^n \ell_i, 
     \label{eq:convergence-inequality-2}
     \\
    e_n 
    & \le 
    \ell_n + \frac{C_2}{n}+ \frac{C_2}{n}\sum_{i=1}^n d_1(\hat\mu^{(i+1)},\hat\mu^{(i+2)}) + \sum_{i=1}^n \frac{i+1}{n} \ell_i, 
    \label{eq:convergence-inequality-1}
\end{align}
for some constants $C_1$ and $C_2$.
\end{theorem}
\begin{proof}[Sketch of Proof]
Our argumentation builds up on the exact FP analysis of~\cite[Theorem 3.3.1]{Hadikhanloo-phdthesis}, which hereby extends to the approximate best response setting. 
\\ 
Let us introduce the approximate exploitability
\be*
    \hat e_n 
    := \J(\hat\pi^{(n+1)},\bar\mu^{(n)}) - \J(\bar\pi^{(n)}, \bar\mu^{(n)}) \,\ge 0\,,
\ee*
 so that $e_n = \ell_n + \hat e_n$, for $n\in\mathbb{N}$. In order to control the exploitability  $e_n$, we  focus our analysis on $\hat e_n$. Denoting $J_n: x \in \cX^T\mapsto  J(x_0, (x_{t+1} - x_t)_{t=0,\dots,T-1},\bar\mu^{(n)})$, we get: % 
\be*
    \hat e_{n+1}-\frac{n}{n+1} \hat e_n 
    &\!\!=&\!\!
     \int_{\cX^T} J_{n+1} d(\hat\mu^{(n+2)} - \bar\mu^{(n+1)}) \\
    % \\
     && \,
    -\frac{n}{n+1} \int_{\cX^T} J_n d(\hat\mu^{(n+1)} -\bar\mu^{(n)})\\
    &\!\!=&\!\!
  n \int_{\cX^T} ( J_{n+1} - J_n )  d(\bar\mu^{(n+1)} - \bar\mu^{(n)})  
  \\
  && \, + \int_{\cX^T} J_{n+1} d(\hat\mu^{(n+2)} - \hat\mu^{(n+1)})\,,
\ee* 
where the last equality follows from the definition of $\bar\mu^{(n)}$.\\
 The monotonicity of the reward in Assumption \ref{Assumption_reward} implies 
\be*%q\label{eq_hat_phi}
\hat e_{n+1}-\frac{n}{n+1}\hat e_n &\le& \int_{\cX^T} J_{n+1} d(\hat\mu^{(n+2)} - \hat\mu^{(n+1)}).
\ee*
Besides, Assumption~\ref{Assumption_reward} together with the compactness of $\cX$ and \cite[Lemma 3.5.2]{Hadikhanloo-phdthesis} and Lemma \ref{Lemma_close_FP} imply that  
$J_{n+1}-J_n$ is $C/n$-Lipschitz, leading to 
\be*
    \hat{e}_{n+1}-\frac{n}{n+1}\hat{e}_n 
    &\le& \int_{\cX^T} J_{n} d(\hat\mu^{(n+2)} - \hat\mu^{(n+1)}) 
     \\
     &&\;
    +\; \frac{C}{n} d_1(\hat\mu^{(n+2)},\hat\mu^{(n+1)})\;.
\ee* 
As $\pi^{*,(n+1)}$ is the best response to the mean field flow $\bar\mu^{(n)}$, recalling the definition of $\ell_n$ in \eqref{Error_Learning}, we deduce
$$
\hat{e}_{n+1}-\frac{n}{n+1}\hat{e}_n 
\le \ell_n + \frac{C}{n} d_1(\hat\mu^{(n+2)},\hat\mu^{(n+1)}) .
$$
Together with estimate $e_n = \ell_n +  \hat e_n$ and \cite[Lemma 3.3.1]{Hadikhanloo-phdthesis}, we derive  \eqref{eq:convergence-inequality-1} and conclude the proof.
\end{proof}

Bound \eqref{eq:convergence-inequality-2} indicates a nice averaging aggregation of the learning errors $(\ell_n)_n$, but requires a strong additional control on the Wasserstein distance between the MF flows generated by both approximate and exact best responses. Such estimate is readily available for the numerical approximation of convex stochastic control problems~\citep{kushner2013numerical} but less classical in the RL literature, as discussed in Sec~\ref{Section_Discussion_RL}. When such an estimate is not available, Bound \eqref{eq:convergence-inequality-1} provides a slower $n\ell_n$ convergence rate, up to a weak $d_1$-regularity of the approximate best response in terms of the mean field flow $\bar\mu^{(n)}$, recall Lemma \ref{Lemma_close_FP}. Such estimate is highly classical in the setting of convex stochastic control problems with Lipschitz rewards \citep{fleming2012deterministic,kushner2013numerical}. 

At finite distance, the following corollary sums up these properties in terms of MF Nash equilibrium. 

\begin{corollary}
\label{cor:main_1}
 Under  Assumption~\ref{Assumption_reward}, if ever $\frac{1}{n}\sum_{i=1}^n\big(\ell_i + C_1 d_1(\mu^{*,(i)},\hat\mu^{(i)})) \big)$ or $\frac{1}{n}\sum_{i=1}^n \big((i+1) \ell_i + C_2 d_1(\hat\mu^{(i+1)},\hat\mu^{(i)})) \big)$ is bounded by $\eps^2/2$, $(\bar\mu^{(n)},\bar\pi^{(n)})$ is an $(\eps,\eps)$-MF Nash equilibrium, for $n$ large enough.
\end{corollary}
In a similar fashion, we can conclude on the general asymptotic convergence of Algorithm \ref{algo:AFP-MFG-2} to the unique MF-Nash equilibrium, before discussing the specific implications for RL best response approximation schemes. \begin{corollary}
\label{cor:main_2}
 Under  Assumption~\ref{Assumption_reward}, the approximate FP algorithm converges to the unique MF Nash equilibrium whenever one of the following two conditions holds:%\\[-6mm]
 \begin{enumerate}
     \item The approximate best response update procedure $\bar\mu^{(n)}\mapsto\hat\mu^{(n+1)}$ is continuous in $d_1$, and $n\ell_n \to 0$, as $n \to \infty$\;;%\\[-6mm]
     \item The learning and policy approximation errors $\ell_n$ and $(d_1(\mu^{*,(n)},\hat\mu^{(n)}))_n$ converge to $0$.
\end{enumerate}
\end{corollary}
The convergence of the sequence $(\bar\mu^{(n)})_n$ follows from the tightness and pre-compactness property of this collection of measures with respect to the Wasserstein distance, see e.g. Remark 3.5.3 in \cite{Hadikhanloo-phdthesis}.

\subsection{Discussion on the convergence for Best Response RL approximation}\label{Section_Discussion_RL} 

The result in Theorem~\ref{Thm_error_1} is general and relies on standard assumptions of MFGs. It also relies on a good enough control of the approximation error on the best response at each iteration. Here, we discuss to what extent existing theoretical results for RL algorithms allow satisfying this assumption.

As stated in Corollary~\ref{cor:main_2}, in order for the approximate FP to converge to the exact MF Nash equilibrium, the approximate best response should converge quickly enough to the best one, depending on the number of iterations. From an RL perspective, this would require being able to compute the approximate optimal policy to an arbitrary precision, with high probability. As far as we know, such a result is possible only when an exact representation of any value function is possible, that is, in the tabular setting which imposes finite state and action spaces. Notably, convergence and rate of convergence of Q-learning-like algorithms have been studied in the literature, see e.g.~\citep{szepesvari1998asymptotic,kearns1999finite,even2003learning,azar2011speedy}. For example, the speedy Q-learning algorithm requires $\mathcal{O}(\ln(K)/(\epsilon^2 (1-\gamma)^4))$ steps to learn an $\epsilon$-optimal state-action value function with high probability, with $K$ the number of state-action couples. According to Corollary~\ref{cor:main_2}, if the error is in $\mathcal{O}(n^{-\alpha})$ with $\alpha>1$ (and if we have continuity in $d_1$), then the scheme converges to the Nash equilibrium. This suggests using $\mathcal{O}(\ln(K)n^{2\alpha}/(1-\gamma)^4)$ steps for the RL agent at iteration $n$. Yet, this kind of results does not provide guarantees on the continuity in $d_1$.

According to Corollary~\ref{cor:main_1}, bounding the learning errors and the distance between two iterates of the distribution is sufficient to reach an approximate Nash equilibrium. As approximate FP can be seen as repeated RL problems, RL (or approximate dynamic programming) can be seen as repeated supervised learning problems, and the propagation of errors from supervised to RL is a well studied field, see e.g.~\citep{farahmand2010error,scherrer2015approximate}. Basically, if the supervised learning steps are bounded by some $\mathcal{O}(\epsilon)$, then the learning error of the RL algorithm is bounded by $\mathcal{O}(C \epsilon / (1-\gamma)^2)$, where $C$ is a so -called concentrability coefficient, measuring the mismatch between some measures. In principle, we could then propagate the learning error of the supervised learning part up to the FP error, through the RL error. However, these results do not provide any guarantees on the proximity between the estimated optimal policy and the actual one (which would be a sufficient condition for the proximity between population distributions); it only provides a guarantee on the distance between their respective returns. This is due to the fact that in RL, the optimal value function is unique, but not the optimal policy. A perspective would be to consider regularized MDPs~\citep{geist2019theory}, where the optimal policy is unique (and greediness is Lipschitz). Yet, this would come at the cost of a bias in the Nash equilibria. The approach in \citep{guo2019learning} somehow builds partially on this idea in their specific learning scheme.

\section{Numerical illustration}\label{Section_Expe}
%\vspace{-0.5em}
 As an illustration, we consider a stylized authoritative  MFG model with congestion in the spirit of \citet{MR3698446}. This application should be seen as a proof of concept showing that the method described above can be applied beyond the framework used for our theoretical results. We compute a model free approximation of the MFG solution combining Algorithm \ref{algo:AFP-MFG-2} with Deep Deterministic Policy Gradient (DDPG)~\citep{lillicrap2015continuous}. As far as we know, this is the first numerical illustration of model free deep RL Algorithm for MFG with continuous states and actions. 
Our numerical results also demonstrate the empirical convergence of the Fictitious RL scheme in a larger setting, even when the MFG is of not first order type.
 
As usual in RL, instead of~\eqref{eq:def-general-J} we consider the problem in infinite horizon with the following discounted reward:
\begin{equation}
    J(x_0,\pi,\mu) := 
    \E\left[ \sum_{t = 0}^{\infty} \gamma^t r(x^{\pi}_t,\mu_t,a_t)  \right]\;, 
\end{equation}
when an infinitesimal player interacts with the population MF flow $\mu = (\mu_0, \mu_1, \mu_2, \dots)$. The goal is to learn the policy which is optimal in the long run, i.e., when the behavior of the population becomes stationary.

\begin{figure*}[ht!]
  \begin{tabular}{@{}cccc@{}}
    \includegraphics[width=.229\textwidth]{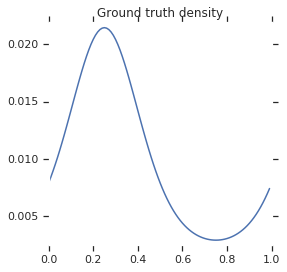} &
    \includegraphics[width=.229\textwidth]{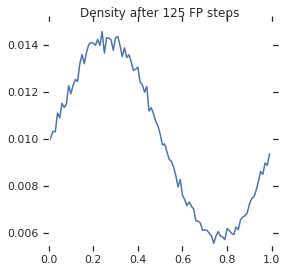} &
    \includegraphics[width=.235\textwidth]{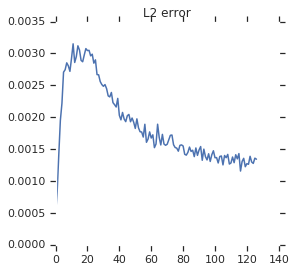} &
    \includegraphics[width=.235\textwidth]{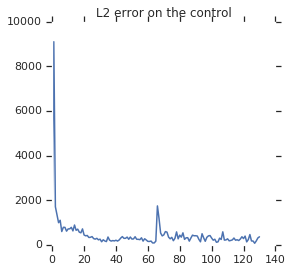} 
  \end{tabular}
%   \vspace{-5.0mm}
  \caption{From left to right: 1) the continuous time explicit solution for $\gamma=1$, 2) the discrete time learned distribution of the policy with $\gamma=0.95$ and $\Delta_t = 0.05$. 3) the $L_2$ error between the explicit solution and the density for each FP iterations, 4) the $L_2$ error between the control and the optimal control and for each FP iterations on a trajectory.}
  %\vspace{-5.0mm}
  \label{tradeoff-experiment}
\end{figure*}
\paragraph{Environment} Each agent has a position $x$ located on the torus $\mathbf{T} = [0,1]$ with periodic boundary conditions (for simplicity of explicit solution), whose dynamics is governed by $x_{t+1} = x_t +  a_t \Delta_t + \sqrt{\Delta_t} \epsilon_t, \quad t = 0, 1, \dots,$ 
 where $\Delta_t$ is the time step of the continuous time process. It receives the per-step reward 
 $$
     r(x_t,\mu_t,a_t) = \tilde r(x_t) -\frac{1}{2}|a_t|^2 - \log(\mu_t)\;,
$$
 where the last term motivates agents to avoid congestion, \textit{i.e.} the proximity to a region with a large population density. In the continuous time setting with no discounting, a direct PDE argument provides the ergodic solution $(a^*,\mu^*)$ in closed form~\citep{MR3698446} 
 \begin{equation}
 \label{eq:sol-expli}
 a^*:x\mapsto \pi \cos(2\pi x)\;\quad \mbox{and} \;\quad 
 \mu^*:x\mapsto  \frac{e^{2 \sin(2\pi x)}}{\int_{\mathbf{T}} e^{2 \sin(2\pi y)} dy}\;,
 \end{equation}
 when the geographic reward is of the form $\tilde r(x_t)=2 \pi^2 \sin(2 \pi x_t) - 2 \pi^2 \cos(2 \pi x_t)^2 + 2 \sin(2 \pi x_t)$. This closed form solution offers a nice benchmark for our experiments and allows to measure the errors made by our algorithm.
 \begin{algorithm}[!h]
\DontPrintSemicolon
\KwData{Number of FP iterations $N_{FP}$, an empty buffer of trajectories $R_{FP}$, the number of trajectories  $N_{\textrm{trajectories FP}}$ to add to the replay buffer $R_{FP}$ per steps}
\KwResult{a density $\bar\mu^{(N_{FP})}$ and a strategy $\hat \pi^{(N_{FP})}$.}
\Begin{
  \For{$n = 1, \dots, N_{FP}$}{
    Compute a best response $\hat\pi^{(n)}$ using DDPG \;
    Collect trajectories $N_{\textrm{trajectories FP}}$ of the strategy $\hat\pi^{(n)}$ and add them to the buffer  $R_{FP}$ \;
    Learn the population mean field distribution $\bar\mu^{(n)}$ by Monte Carlo samples on $R_{FP}$ \; 
    Compute the uniformly randomized policy $\bar \pi^{(n)}$ on $(\hat \pi^{(1)},\ldots,\hat \pi^{(n)})$\;
    }
  \KwRet{$\bar\mu^{(N_{FP})}$ and $\bar \pi^{(N_{FP})}$}
  }
\caption{Fictitious Play for continuous state and action Mean Field Games}
\label{algo:AFP-MFG-Experiment}
\end{algorithm}
\paragraph{Implemented Algorithm} 
Model free FP for MFGs takes a somehow similar approach as~\citet{Lanctot17PSRO} in the sense that we estimate the best response using a model free RL algorithm (namely DDPG). However we do not maintain those best responses as in~\citep{Lanctot17PSRO} but rather learn the population MF flow of the distribution of the representative agents. 
The best response approximation through DDPG and the estimation of the population MF are left in the Algorithm. We ran $30000$ trajectories of DDPG with a trajectory length of $300$. The noise used for exploration is a centered normal noise with variance $0.02$ and we used Adam optimizers with $0.001$ starting learning rate and $\tau = 0.01$. At each iteration of FP, we added $N_{\textrm{trajectories FP}} = 3000$ trajectories of length $1000$ to the replay buffer. Finally, we estimated the density using $100$ classes and doing $30000$ steps of Adam (with $0.001$ initial learning rate). 
\paragraph{Results.} Figure \ref{tradeoff-experiment} presents the learned equilibrium computed for $\gamma=0.95$, $N=90$ and uniform initial distribution, as well as the continuous time closed form ergodic solution for $\gamma=1$, see~\eqref{eq:sol-expli}. We emphasize that the variation in $\gamma$ together with the discrete/continuous time difference setting implies that the theoretical solutions to both problems are close but do not exactly coincide. We keep this benchmark since no ergodic closed form solution is available for $\gamma\neq 1$. As observed on Figure \ref{tradeoff-experiment}, both ergodic explicit and learned distributions and controls are close. As expected, the density of players is larger around the point of maximum of the reward but the distribution is not highly concentrated due to the logarithmic penalty encoding aversion for congested regions. More precisely, Figure \ref{tradeoff-experiment} indicates that the $L_2$ errors between the distributions and the controls decrease with the number $n$ of iterations. The convergence of control distributions echoes to the discussion on error propagation in Section \ref{Sec_Error_Propagate}. This clearly illustrates the numerical convergence of the Deep RL FP mean field algorithm. 
\section{Related work}
The related literature is as follows. Recently, model free RL algorithms for solving MFGs were analyzed in the following papers: \citet{guo2019learning} and~\citet{tiwari2019reinforcement} study $Q$-learning, \citet{mguni2018decentralisedli}
consider FP but contains several inaccuracies, as already pointed out in~\citet{subramanianpolicy}, which focuses on policy gradient methods. However, their studies are restricted to a stationary setting and focus on particular RL algorithms. Their convergence results hold under assumptions that are often hard to verify in practice. Although not focusing on an MFG, \citet{yang2018mean} uses the idea of MF approximation by considering interactions  through the empirical mean action. Numerical illustrations provided in all these papers are in a finite state-action setting, while we present a numerical example in a continuous  state-action setting. On a different note, \citet{yang2018deep} studies the link between MFG and \emph{inverse} RL. Some authors also study ``learning'' algorithms which use the full knowledge of the model (and hence are not model-free): \citet{yin2010learning} studied a MF oscillator game while \citet{hu2019deep} proposed a decentralized deep FP learning architecture for large MARL, whose convergence holds on linear quadratic MFG examples with explicit solution and small maturity. 
\section{Conclusion and future research}
 In comparison to the existing literature focusing on specific RL algorithms for MFGs, we took a step back and offer a general perspective on the error propagation in iterative scheme for MFG, using any learning algorithm. We presented a rigorous convergence analysis of model free FP learning algorithm for MF Agent systems, encompassing cases where  the best response is approximated using any single agent learning algorithm as well as non-stationary settings. We showed how the convergence of model free iterative FP algorithm reduces to the error analysis of each learning iteration step, as the convergence of RL algorithm reduces to the aggregation of repeated supervised learning  approximation errors \citep{farahmand2010error,scherrer2015approximate}. Our theoretical setting covers for the first time the consideration of non-stationary MFG and relies on reasonable and verifiable assumptions on the MFG of interest. The convergence is illustrated for the first time by numerical experiments in a continuous state-action setting, based on deep RL algorithm. Our analysis motivates and properly justifies the use of asymptotic Mean Field approximation for the study of learning by experience schemes in Multi-Agent systems, with a large number of agents.\\
For RL approximation schemes, our analysis suggests a much faster convergence rate, whenever the best response approximation quality can be controlled in Wasserstein distance. This kind of estimate is  classical in the numerical approximation of stochastic control literature but currently not available in the RL literature. The derivation of such estimate  deserves to be addressed in future research papers. Finally, we focused on convergence properties for a centralized Multi-Agent learning algorithm, paving the way for addressing such property for a more relevant decentralized one.

\clearpage
\newpage

{\fontsize{9.0pt}{10.0pt}
\selectfont
\bibliographystyle{aaai}
\bibliography{afp-mfg_bib}
}

\clearpage

\onecolumn

\section*{Appendix}
This Appendix regroups the technical proofs related to the error propagation bounds on the Approximate Fictitious Play algorithm detailed in Theorem \ref{Thm_error_1}. Many arguments reported here are inspired by the results presented in \citep{Hadikhanloo-phdthesis} for the exact fictitious play algorithm.

We follow the notations of Section \ref{Section_First_Order}. In particular, we recall that $\cX$ denotes the set of accessible states before time $T$. Since we take $b(x,a,\mu) = a$ and the set $\cA$ is compact, $\cX^T = \cX \times \dots \times \cX$ is compact too and in particular it is a bounded subset of $\R^{T \times d}$.

In order to measure the proximity between MF population flows, we denote by $d_1$, the $1-$ Wasserstein distance defined (using Kantorovitch-Rubinstein duality) as: for all $\mu, \mu' \in M_T$,
$$
    d_1(\mu, \mu')
    =
    \sup_{h \in Lip_1(\cX^T, \R)} \sum_{t \in \Tb} \int_{\cX} h(x) d(\mu_t - \mu'_t)(x), 
$$
where $Lip_1(\cX^T, \R)$ is the set of $1$-Lipschitz continuous function from $\cX^T$ to $\R$.

\subsection*{A. Stability of the FP mean field flow $(\bar\mu^{(n)})_n$}
Let us first provide the proof of Lemma \ref{Lemma_close_FP} which ensures the closeness in $d_1$ of two consecutive elements of the Mean field flow learning sequence $\bar\mu^{(n)}$.  Let first recall from the definition of $\bar\mu^{(n)}$ that we have: for all $n\in\mathbb{N}$,
\begin{align}
    \bar\mu^{(n+1)}-\bar\mu^{(n)} 
    &= 
    \frac{1}{n}\left[\hat\mu^{(n+1)}-\bar\mu^{(n+1)}\right]
    \notag
    \\
    \label{eq:FP-rel-barmu-hatbar}
    &= 
    \frac{1}{n+1}\left[\hat\mu^{(n+1)}-\bar\mu^{(n)}\right] \;.
\end{align}

\begin{proof}[Proof of Lemma \ref{Lemma_close_FP}]
Let $h \in Lip_1(\cX^T, \R)$. We recall that $\cX^T$ is bounded and pick $\mathbf{x}_0\in\cX^T$. Then, using~\eqref{eq:FP-rel-barmu-hatbar} together with the definition of $\bar\mu^{(n)}$, we compute 
\begin{align*}
    \left|\int_{\cX^T} h(\mathbf{x}) d(\bar\mu^{(n+1)} - \bar\mu^{(n)})(\mathbf{x})\right|
    &= 
    \frac{1}{n+1} \left|\int_{\cX^T} h(\mathbf{x}) d(\hat\mu^{(n)} - \bar\mu^{(n)})(\mathbf{x})\right|
    \\
    &= \frac{1}{n+1} \left|\int_{\cX^T} (h(\mathbf{x}) - h(\mathbf{x}_0)) d(\hat\mu^{(n)} - \bar\mu^{(n)})(\mathbf{x})\right|
    \\
    &\le \frac{1}{n+1} \int_{\cX^T} \|\mathbf{x}-\mathbf{x}_0\| \left[ d \hat\mu^{(n)}(\mathbf{x}) +  d \bar\mu^{(n)}(\mathbf{x}) \right] 
    \\
    &\le \frac{C}{n+1}\;,
\end{align*}
since $\cX^T$ is bounded. This result being valid for any $h \in Lip_1(\cX^T, \R)$, we obtain $$d_1(\bar\mu^{(n+1)}, \bar\mu^{(n)}) \le \frac{C}{n+1} \;, \quad n\in\mathbb{N}\,.$$ 
\end{proof}

\subsection*{B. Propagation error estimates}

This section is dedicated to the rigorous derivation of the bounds \eqref{eq:convergence-inequality-1} and \eqref{eq:convergence-inequality-2}  presented in Theorem \ref{thm_approximate_Nash}.

We first recall the following useful result, see e.g.~\cite[Lemma 3.3.1]{Hadikhanloo-phdthesis}.
\begin{lemma}
\label{lem:useful-bound-sequence}
Let $(\varphi_n)_{n}$ and $(\lambda_n)_{n}$ be two sequences of real numbers such that 
$$
    (n+1)\varphi_{n+1}-n\varphi_n \le {\lambda_n}, \qquad n\in\mathbb{N}\,.
$$
Then, we have the estimate: 
$$
    \varphi_n \le \frac{\varphi_0}{n} + \frac{1}{n} \sum_{i=1}^n \lambda_i, \qquad n\in\mathbb{N}\;.
$$ 
\end{lemma}

For ease of notation, we introduce $J_n := J(.,\bar\mu^{(n)})$ and $J_{n+1} := J(.,\bar\mu^{(n+1)})$ (which are functions defined over $\cX^T$), for all $n \geq 0$. More precisely, for $x \in \cX^T$,
$$
    J_n(x) := J(x_0, (x_{t+1} - x_t)_{t=0,\dots,T-1},\bar\mu^{(n)}).
$$
Observe that this definition is accurate by the definition of the first order MFG setting presented in Section \ref{Section_FirstOrderMFG} and because there is a bijection between process trajectory and the combination of initial position and policy.

\begin{proof}[Proof of estimate \eqref{eq:convergence-inequality-1} in Theorem \ref{Thm_error_1}]
We adapt the arguments in the proof of~\cite[Theorem 3.3.1]{Hadikhanloo-phdthesis} to our setting with approximate best responses.

Let us introduce the approximate learning error $\hat e_n$ defined by: for $n\in\mathbb{N}$,
\begin{align*}
    \hat e_n 
    &:= \E_{x_0 \sim \mu_0}[J(x_0, \hat\pi^{(n+1)}, \bar\mu^{(n)})-J(x_0, \bar\pi^{(n)}, \bar\mu^{(n)})]\ge 0\;,
\end{align*}
so that $e_n = \ell_n + \hat e_n$. In order to control $e_n$, we will focus our analysis on $\hat e_n$ and compute
\be*
    &&(n+1)\hat e_{n+1}-n \hat e_n 
    \\
    &=& 
    (n+1) \int_{\cX^T} J_{n+1} d(\hat\mu^{(n+2)} - \bar\mu^{(n+1)}) 
    -{n} \int_{\cX^T} J_n d(\hat\mu^{(n+1)} -\bar\mu^{(n)})\\
    &=&
    (n+1)\int_{\cX^T} J_{n+1} d(\hat\mu^{(n+1)} - \bar\mu^{(n+1)}) 
    -n\int_{\cX^T} J_n d(\hat\mu^{(n+1)} - \bar\mu^{(n)})
   +\quad (n+1)\int_{\cX^T} J_{n+1} d(\hat\mu^{(n+2)} - \hat\mu^{(n+1)})
    \\
    &=&
  n(n+1) \int_{\cX^T} ( J_{n+1} - J_n )  d(\bar\mu^{(n+1)} - \bar\mu^{(n)})  
   + (n+1)\int_{\cX^T} J_{n+1} d(\hat\mu^{(n+2)} - \hat\mu^{(n+1)}),
\ee* 
 where the last equality follows from (\ref{eq:FP-rel-barmu-hatbar}).\\
 Thanks to Assumption \ref{Assumption_reward}, the monotonicity of the reward function implies directly 
\beq\label{eq_hat_phi}
(n+1)\hat e_{n+1}-n\hat e_n \le (n+1)\int_{\cX^T} J_{n+1} d(\hat\mu^{(n+2)} - \hat\mu^{(n+1)}).
\eeq 
 By definition of $J_n$ together with the first order MFG dynamics, we have the expression
$$
    J_n(x) = \sum_{t=0}^{T-1} \gamma^t\left[\tilde r(x_t, x_{t+1} - x_{t}) + \bar r(x_t, \bar\mu^{(n)}_t)\right].
$$
Moreover, using Assumption~\ref{Assumption_reward} and the compactness of $\cX$, we deduce as in~\cite[Lemma 3.5.2]{Hadikhanloo-phdthesis}, the existence of a constant $C$ such that for all $\mathbf{x}, \mathbf{x}' \in \cX^T$,
\begin{equation}
\label{eq:Jn-lip-property}
\begin{split}
     &|J_{n+1}(\mathbf{x}) - J_n(\mathbf{x}) - J_{n+1}(\mathbf{x}') + J_n(\mathbf{x}')|
    \le C \|\mathbf{x}-\mathbf{x}'\|_\infty \, d_1(\bar\mu^{(n+1)},\bar\mu^{(n)}).
\end{split}
\end{equation}

This property of the reward function together with  Lemma~\ref{Lemma_close_FP} indicate that $J_{n+1}-J_n$ is $C/n$-Lipschitz, so that  
\be*
    && (n+1)\hat{e}_{n+1}-n\hat{e}_n 
    \\
    &\le& (n+1)\int_{\cX^T} J_{n} d(\hat\mu^{(n+2)} - \hat\mu^{(n+1)}) 
    + C d_1(\hat\mu^{(n+2)},\hat\mu^{(n+1)})
    \\
    &\le& (n+1)\int_{\cX^T} J_{n} d(\mu^{*,(n+1)} - \hat\mu^{(n+1)}) 
    + C d_1(\hat\mu^{(n+2)},\hat\mu^{(n+1)})\;,
\ee* 
where, in the last inequality, we used the fact that $\pi^{*,(n+1)}$ is the best response with respect to $\bar\mu^{(n)}$ and hence
\begin{align}
\label{eq:proof-cv-ineq-mustar}
    \int_{\cX^T} J_{n} d \mu^{*,(n+1)}
    &\geq
    \int_{\cX^T} J_{n} d \hat\mu^{(n+2)}.
\end{align}
Indeed, by optimality of $\pi^{*,(n+1)}$, we have
\begin{align*}
    \int_{\cX^T} J_{n} d \mu^{*,(n+1)}
    &= 
    \J(\pi^{*,(n+1)}, \bar\mu^{(n)}) \,\geq\;
    \J(\hat \pi^{(n+2)}, \bar\mu^{(n)}).
\end{align*}
 By definition of the learning error $\ell_{n}$ in \eqref{Error_Learning}, we deduce
$$
(n+1)\hat{e}_{n+1}-n\hat{e}_n 
\le (n+1)\ell_n + C d_1(\hat\mu^{(n+2)},\hat\mu^{(n+1)}) .
$$

By Lemma~\ref{lem:useful-bound-sequence} applied to 
$$
    \varphi_n = \hat{e}_n, 
    \quad \hbox{ and } \quad
    \lambda_n = (n+1) \ell_{n} + C d_1(\hat\mu^{(n+2)},\hat\mu^{(n+1)}),
$$ 
we derive the estimate 
$$
    \hat e_n \le \ell_n + \frac{\hat e_0}{n}+\frac{1}{n}\sum_{i=1}^n (i+1)\ell_i +\frac{1}{n}\sum_{i=1}^n d_1(\hat\mu^{(i+2)},\hat\mu^{(i+1)}).
$$
Combining this estimate with the relation $e_n = \ell_n +  \hat e_n$ provides \eqref{eq:convergence-inequality-1}.
\end{proof}

We are now in position to turn to the proof of the remaining estimate \eqref{eq:convergence-inequality-2}.

\begin{proof}[Proof of estimate \eqref{eq:convergence-inequality-2} in Theorem \ref{Thm_error_1}]
 Combining the relation $e_n = \ell_n +  \hat e_n$ together with estimate \eqref{eq_hat_phi}, we compute
\be*
&& (n+1)e_{n+1}-ne_n \\
&\le&
 (n+1)\int_{\cX^T} J_{n+1} d(\mu^{*,(n+2)} - \hat\mu^{(n+1)}) 
 - n \int_{\cX^T} J_n d(\mu^{*,(n+1)} - \hat\mu^{(n+1)}) 
 \\
&=&
 (n+1)\int_{\cX^T} J_{n+1} d(\mu^{*,(n+2)} - \mu^{*,(n+1)}) 
 + \int_{\cX^T} [(n+1)J_{n+1} - n J_n] d(\mu^{*,(n+1)} - \hat\mu^{(n+1)}) 
 \\
&=&
 (n+1)\int_{\cX^T} [J_{n+1}-J_n] d(\mu^{*,(n+2)} - \mu^{*,(n+1)}) 
 + (n+1)\int_{\cX^T} J_n d(\mu^{*,(n+2)} - \mu^{*,(n+1)})
 \\ &&\quad +   (n+1)\int_{\cX^T} [J_{n+1} - J_n] d(\mu^{*,(n+1)} - \hat\mu^{(n+1)}) 
 + \int_{\cX^T} J_n d[\mu^{*,(n+1)}-\hat\mu^{(n+1)}]\,.
\ee* 

 Hence, the Lipschitz property~\eqref{eq:Jn-lip-property} together with Lemma  \ref{Lemma_close_FP} implies 
\be*
 (n+1)e_{n+1}-ne_n 
    &\le &
    C d_1(\mu^{*,(n+2)},\mu^{*,(n+1)}) + C d_1(\mu^{*,(n+1)},\hat\mu^{(n+1)}) + \ell_n\;,
\ee*
where we also used that $\mu^{*,(n+1)}$ is the mean field flow  induced by the best response to the population distribution $\bar\mu^{(n)}$ so that 
\begin{align*}
    \int_{\cX^T} J_{n} d \mu^{*,(n+1)}
    &\geq
    \int_{\cX^T} J_{n} d \hat\mu^{(n+2)}.
\end{align*}

Finally, Lemma \ref{Lemma_close_FP} together with the continuity of the best response in the first order MFG {(which stems from Assumption~\ref{Assumption_reward}, see e.g.~\cite[Remark 3.5.3]{Hadikhanloo-phdthesis})} ensure that $d_1(\mu^{*,(n+2)},\mu^{*,(n+1)})$ converges to zero as $n$ goes to infinity. Finally, applying Lemma~\ref{lem:useful-bound-sequence} to 
$$
    \varphi_n = e_n, 
    \quad \hbox{ and } \quad
    \lambda_n =  \ell_{n} + C d_1(\mu^{*,(n+1)},\hat\mu^{(n+1)})
    + \frac{C}{n}
$$  directly concludes the proof.
 \end{proof}
\newpage 

\subsection*{C. Algorithms}

For sake of completeness, we detail here the pseudo code for important algorithms of the paper:  first exact fictitious play, then Deep Deterministic Policy Gradient.% and finally the density approximation scheme. 
 \begin{algorithm}[!h]
 \DontPrintSemicolon
\KwData{An initial distribution $\mu_0$; an initial mean field flow $\bar{\mu}^{*,(0)}$; a number of steps $N$.}
\KwResult{A mean field flow $\mu$ and a  policy $\pi$.}
\Begin{
  \For{$n = 0, 1, \dots N-1$}{
    Compute $\pi^{*,(n+1)}$, the best response policy against $\overline{\mu}^{*,(n)}$ (see Definition \ref{def:BR})\;
    Update  $\bar{\mu}^{*,(n+1)} =\frac{n}{n+1}{\bar \mu}^{*,(n)} + \frac{1}{n+1} \mu^{*,(n+1)}$  \;
    }
      \KwRet{$(\bar\mu^{*,(N)}, \pi^{*,(N)})$} 
  }
\caption{Exact Fictitious Play for MFG}
\label{algo:EFP-MFG-2}
\end{algorithm}
\label{DDPG-Appendix}
\begin{algorithm}[h!]
\DontPrintSemicolon
\KwData{Randomly initialize critic network $Q(x, a| \theta^Q)$ and actor network $\pi(x|\theta^\pi)$ with weights $\theta^Q$ and $\theta^\pi$ and initialize target network $Q'$ and actor network $\pi'$ with weights $\theta^{Q'} \leftarrow \theta^Q$ and $\theta^{\pi'} \leftarrow \theta^\pi$.}
\KwResult{a policy $\pi$.}
\Begin{
  \For{ episode $=0, 1, \dots M-1$}{
    Initialize replay buffer $R$ \;
    \For{$t=0, 1, \dots T-1$}{
    Select an action $a_t = \pi(x|\theta^\pi) + \mathcal{N}_t$ according to the current policy and exploration noise,\;
    Execute action $a_t$ and observe reward $r_t$ and new state $x_{t+1}$\;
    Store transition $(x_t, a_t, r_t, x_{t+1})$ in $R$ \;
    Sample a random minibatch of $N$ transitions $(x_i, a_i, r_i, x_{i+1})$ from $R$ \;
    Set $y_i = r_i + \gamma Q'(x_{i+1}, \pi'(x_{i+1}|\theta^{\pi'}))$ \;
    Update critic by minimizing the loss: $L(\theta^Q)=\frac{1}{N}\sum \limits_i (y_i - Q(x_i, a_i|\theta^Q))^2$ \;
    Update the actor policy using the sampled policy gradient: $\nabla_{\theta^\pi} J \simeq \frac{1}{N}\sum \limits_i \nabla_a Q(x,a|\theta^Q)|_{x=x_i, a=\pi(x_i)} \nabla_{\theta^\pi} \pi(x|\theta^\pi)|_{x=x_1}$ \;
    Update target networks:
    $$\theta^{Q'} \leftarrow \tau \theta^{Q} + (1-\tau)\theta^{Q'}$$
    $$\theta^{\pi'} \leftarrow \tau \theta^{\pi} + (1-\tau)\theta^{\pi'}$$
    }
    }
  \KwRet{$\pi$}
  }
\caption{DDPG}
\label{algo:AFP-MFG-DDPG}
\end{algorithm}
\subsection*{D. Approximation of the density}
The density estimation is done through classification. We divide a state dataset $X$ of size $N$ into classes representing a partitioning of the space $\left([\frac{i}{N_{\textrm{classes}}}, \frac{i+1}{N_{\textrm{classes}}}]\right)_{i\in\{0,..., N_{\textrm{classes}}-1\}}$. Then we use a function $f_{\theta} \in \{0\} \rightarrow \Delta_{N_{\textrm{classes}}}$ to estimate the density of $x \in X$ being in this interval $[\frac{i}{N_{\textrm{classes}}}, \frac{i+1}{N_{\textrm{classes}}}]$ by minimizing the cross entropy loss $L(\theta)=\frac{1}{N}\sum \limits_{i=1}^N \sum \limits_{j=1}^{N_{\textrm{classes}}} \mathds{1}_{x_i \in [\frac{j}{N_{\textrm{classes}}}, \frac{j+1}{N_{\textrm{classes}}}]} \log(f_{\theta}^j)$.

\newpage 
\subsection*{E. Additional numerical results}

 We provide in this Section  additional numerical results on the illustrative example presented in Section \ref{Section_Expe}. Figure \ref{tradeoff-experiment2} presents the ergodic continuous time theoretical optimal policy together with the numerically estimated optimal one for the  discrete time model, with parameters detailed in Section \ref{Section_Expe}.
\begin{figure*}[h!]
\centering
  \begin{tabular}{@{}cccc@{}}
     \includegraphics[width=.30\textwidth]{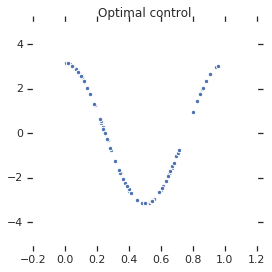} &
     \includegraphics[width=.30\textwidth]{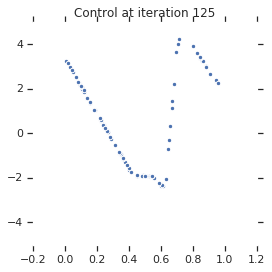} &
  \end{tabular}
  \vspace{-5.0mm}
  \caption{From left to right: 1) the continuous time explicit solution policy for $\gamma=1$, 2) the discrete time learned distribution of the policy with $\gamma=0.95$ and $\Delta_t = 0.05$.
  }
  \vspace{-5.0mm}
  \label{tradeoff-experiment2}
\end{figure*}
\end{document}